\documentclass[twoside, 12pt]{article}
\usepackage[utf8]{inputenc}
\usepackage[T1]{fontenc}
\usepackage{amssymb,amsmath,amsthm,mathtools,mathrsfs}%, wasysym}
\usepackage{enumitem}
\usepackage[colorlinks,bookmarks,linkcolor=black,citecolor=black]{hyperref}
\usepackage{graphicx}
\usepackage{color}
\usepackage[top=2cm, bottom=2cm, left=2cm, right=2cm]{geometry}
\usepackage{float}
\newcommand{\bd}{\begin{description}}
\newcommand{\ed}{\end{description}}
\newcommand{\bi}{\begin{itemize}}
\newcommand{\ei}{\end{itemize}}
\newcommand{\be}{\begin{enumerate}}
\newcommand{\ee}{\end{enumerate}}
\newcommand{\beq}{\begin{equation}}
\newcommand{\eeq}{\end{equation}}
\newcommand{\beqs}{\begin{eqnarray*}}
\newcommand{\eeqs}{\end{eqnarray*}}

\definecolor{DarkGreen}{rgb}{0.2, 0.6, 0.3}

\newtheorem{theorem}{Theorem}

\newtheorem{lemma}{Lemma}

\newtheorem{claim}{Claim}

\newtheorem{question}{Question}

\def\endofClaim{\hfill\scalebox{.6}{$\blacksquare$}}
\newcommand{\oldqed}{}
\newenvironment{claimproof}[1][Proof]{
  \renewcommand{\oldqed}{\qedsymbol}
  \renewcommand{\qedsymbol}{\endofClaim}
  \begin{proof}[#1]
}{
  \end{proof}
  \renewcommand{\qedsymbol}{\oldqed}
}

\begin{document}
\title{Ramsey size linear and generalization}

\author{Eng Keat Hng\thanks{Extremal Combinatorics and Probability Group (ECOPRO), Institute for Basic Science (IBS), Daejeon, South Korea. Supported by IBS-R029-C4. {\tt hng@ibs.re.kr}}
\and
Meng Ji\footnote{Corresponding author: School of Mathematical Sciences, and Institute of Mathematics and Interdisciplinary Sciences, Tianjin Normal University, Tianjin, China. Supported by the Tianjin Municipal Education Commission Scientific Research Program Project (Grant No. 2025KJ133) and the Institute of Basic Science (IBS-R029-C4).  {\tt
mji@tjnu.edu.cn}}
\and
Ander Lamaison\footnote{Universidad
Pública de Navarra, Pamplona, Spain. Supported by the Institute of Basic Science (IBS-R029-C4).  {\tt
lamaison@mail.muni.cz}}
}
\date{January 8, 2026}
\maketitle

\begin{abstract}
More than thirty years ago, Erd\H{o}s, Faudree, Rousseau, and Schelp posed a fundamental question in extremal graph theory: What is the optimal constant $c_k$ such that 
$r(C_{2k+1}, G) \le c_k m$ for any graph $G$ with $m$ edges and no isolated vertices? In this paper,
we make a significant step towards answering this question by proving that 
$r(C_{2k+1}, G) \le (2 + o(1)) m + p,$
where $p$ denotes the number of vertices in $G$. Additionally, we extend the work of
Goddard and Kleitman and independently Sidorenko, who proved that 
$r(K_3, G) \le 2m + 1$
for any graph $G$ with $m$ edges and no isolated vertices. We generalize their findings to the clique version, establishing that 
$r(K_r, G) \le c_r m^{(r-1)/2}$,
and to the multicolor setting, showing that 
$r_{k+1}(K_3; G) \le c_k m^{(k+1)/2}.$
\\[2mm]
\end{abstract}

\section{Introduction}
Given two graphs $H$ and $G$, the \emph{Ramsey number} $r(H,G)$ is the least positive integer $N$ such that every red-blue edge coloring of $K_N$ contains a red copy of $H$ or a blue copy of $G$. Ramsey theory is a fundamental area of combinatorics that studies the conditions under which order must appear within chaos. Understanding and studying Ramsey numbers is one of the central problems in combinatorics, and while much progress has been made in special cases, the general problem remains largely unsolved. We refer the reader to the dynamic survey of Radziszowski~\cite{Radziszowski}.

In the early 1980s, Harary conjectured that for any graph $G$ with $m$ edges and no isolated vertices, the Ramsey number $r(K_{3}, G)$ is at most $2m+1$. This conjecture was later proven by Goddard and Kleitman \cite{Goddard-Kleitman}, and independently Sidorenko \cite{Sidorenko} in the 1990s. This result is tight, as equality is attained when $G$ is a tree or a matching. 
Building on this, Erd\H{o}s, Faudree, Rousseau and Schelp \cite{Erdos-Faudree-Rousseau-Schelp} introduced the concept of Ramsey size linear in 1993. 
A graph $H$ is said to be \emph{Ramsey size linear} if there is a constant $C$ such that for any graph $G$ with $m$ edges and no isolated vertices, the Ramsey number $r(H, G)$ is bounded above by $C\cdot m$. In other words, the Ramsey number grows linearly with the number of edges in $G$. 

The study of Ramsey size linear graphs has led to several important results. Erd\H{o}s, Faudree, Rousseau, and Schelp \cite{Erdos-Faudree-Rousseau-Schelp} proved that every graph $H$ with $p$ vertices and $q\geq 2p-2$ edges is not Ramsey size linear and every connected graph $H$ with $p$ vertices and $q\leq p + 1$ edges is Ramsey size linear. Furthermore, they showed that both bounds on $q$ are sharp. Several special classes of graphs were shown to be Ramsey size linear and bounds on the Ramsey numbers were determined; see \cite{Balister-Schelp-Simonovits, Bradac-Gishboliner-Sudakov,Erdos-Faudree-Rousseau-Schelp,  Jayawardene-Rousseau-Bollobas}. 

Erd\H{o}s, Faudree, Rousseau and Schelp \cite{Erdos-Faudree-Rousseau-Schelp} studied the case that $H$ is a cycle, and proved that \emph{if $k\geq2$ and $G$ is a connected graph of size $m$, then for $m$ sufficiently large $r(C_{2k},G)\leq m+22k\sqrt{m}$.} However, analogue upper bound has not been obtianed for the odd cycle, so they posed the following question, which appears as Problems 569 on Bloom’s Erd\H{o}s
problems website \cite{Erdos-1}.
\begin{question}[\cite{Erdos-1,Erdos-Faudree-Rousseau-Schelp}]\label{Q1}
What is the best possible $c_{k}$ such that $r(C_{2k+1},G)\leq c_{k}m$ for any graph $G$ with $m$ edges without isolated vertices?
\end{question}
Recall that the case $k=1$ of Question \ref{Q1} has been solved by
Goddard and Kleitman \cite{Goddard-Kleitman}, and Sidorenko \cite{Sidorenko} independently. 
\begin{theorem}[\cite{Goddard-Kleitman,Sidorenko}]\label{thm1}
For any graph $G$ with $m$ edges and no isolated vertices, the Ramsey number $r(K_{3}, G)$ is at most $2m+1$.
\end{theorem}

As the first result of this paper, we obtain an upper bound on the Ramsey number $r(C_{2k+1},G)$. 
\begin{theorem}\label{cycle-epsion-p2}
For every $k\ge2$ there is a constant $B_k$ such that for any graph $G$ on $p$ vertices with $m$ edges and no isolated vertices we have
\[r(C_{2k+1},G)\leq 2m\left(1+B_km^{-1/20}\right)+p\;.\]
\end{theorem}

We study two kinds of generalizations of Theorem \ref{thm1}, and first prove the following clique version.

\begin{theorem}\label{multi-kr}
For every $r\ge3$ there is a constant $c_r$ such that for any graph $G$ with $m$ edges and without isolated vertices we have
$$r(K_{r},G)\leq c_{r}m^{\frac{r-1}{2}}.$$
\end{theorem}

Secondly, we derive the multicolor setting of Theorem \ref{thm1}. Given graphs $K_{3}$ and $G$, the \emph{multicolor Ramsey number} $r_{k+1}(K_{3};G)$ is the least $N$ such that a copy of $K_{N}$ edge-colored with $k+1$ colors contains a monochromatic $K_{3}$ in one of the first $k$ colors or a monochromatic $G$ in the last color. 

\begin{theorem}\label{multi-k3}
For every $k\ge1$ there is a constant $c_k$ such that for any graph $G$ with $m$ edges and without isolated vertices we have
\[r_{k+1}(K_{3};G)\leq c_{k}m^{\frac{k+1}{2}}.\]
\end{theorem}
\noindent {Note that the Ramsey number on $K_{3}$ versus $G$ is $r_{2}(K_{3};G)=r(K_{3},G)$. }

\noindent {\bf Remark.} After submitting our paper on January 8, 2026, we discovered that \cite{Cambie-Freschi-Morawski-Petrova-Pokrovskiy} independently obtained a result similar to Theorem 2 on arXiv which solved Problems 570 of Bloom’s Erdős problems.
\section{Proofs}\label{sec2}
In this section, we give the proofs of Theorems 
\ref{cycle-epsion-p2}, \ref{multi-kr} and \ref{multi-k3}. Before proving the above three theorems, we shall show a useful lemma, whose proof needs the following known results.

\begin{theorem}[\cite{Chvatal}]\label{chavtal}
For every tree $T$ on $n$ vertices and every positive integer $p$, we have
\[r(T,K_{p})=(n-1)(p-1)+1.\]
\end{theorem}

\begin{theorem}[\cite{Erdos-Gallai}]\label{E-G-thm}
For $p\geq k\geq 3$ every graph on $p$ vertices with at least $(k-1)(p-1)/2$ edges contains a path of length $k-1$.
\end{theorem}

Now we prove the lemma. Here we denote the path on $n$ vertices by $P_{n}$.

\begin{lemma}\label{path--}
For every integer $\ell\ge4$, every real number $\beta\ge3$ and every graph $G$ on $p$ vertices with $m$ edges and no isolated vertices, we have
\[r(P_{\ell},G)\leq\left(1+\frac{1}{\beta}\right)p+(\beta+1)^{2}\ell\sqrt{m}\;.\]
\end{lemma}
\begin{proof}
Let $N=\left(1+\frac{1}{\beta}\right) p+(\beta+1)^{2}\ell\sqrt{m}$. Take a red/blue-colored copy $\mathbf{K}$ of $K_{N}$ and suppose that it contains no red copy of $P_{\ell}$. Writing $d_{red}(v;\mathbf{K})$ for the red degree of $v\in V(\mathbf{K})$ in $\mathbf{K}$, let
\[L=\left\{u\in V(\mathbf{K})~|~d_{red}(u;\mathbf{K})\geq (\beta+1)\ell\right\}\;.\]
Note that $|L|\le\frac{N}{\beta+1}$. Indeed, suppose not. Then the red subgraph of $\mathbf{K}$ has at least $\ell N/2$ edges. Hence, it contains a red copy of $P_{\ell}$ by Theorem~\ref{E-G-thm}, giving a contradiction. Let $\mathbf{H}$ be obtained from $\mathbf{K}$ by deleting the vertices in $L$. Note that $|V(\mathbf{H})|\ge\frac{\beta N}{\beta+1}$ and for all $v\in V(\mathbf{H})$ we have red degree $d_{red}(v;\mathbf{H})\le(\beta+1)\ell$ in $\mathbf{H}$.

We shall show that $\mathbf{H}$, and so $\mathbf{K}$, contains a blue copy of $G$. First we enumerate the vertices of $G$ in decreasing degree order as $v_1,\dots,v_p$, that is, we have $d_{G}(v_{1})\geq d_{G}(v_{2})\geq \ldots \ge d_{G}(v_{p})$. For $1\le r\le p$ let $G_{r}=G[\{v_{1},\ldots,v_{r}\}]$. Now suppose for a contradiction that $\mathbf{H}$ contains a blue copy $\mathbf{L}$ of $G_{r}$ but not a blue copy of $G_{r+1}$. Write $W$ for set of vertices in $\mathbf{L}$ corresponding to the neighbors of $v_{r+1}$ in $G_{r+1}$. Since $\mathbf{H}$ contains no blue copy of $G_{r+1}$, we immediately have the following observation.
\begin{equation}\label{observe}
\text{\it{Each}}~ v\in V(\mathbf{H})\setminus V(\mathbf{L})~\text{\it{is joined to}}~W~\text{\it{by at least one red edge}.}\tag{A}
\end{equation}
To finish the proof, we consider the following two cases.
\vskip 2mm
{\bf Case 1.} 
$1\le r\leq \frac{p}{2}$.

The graph $\mathbf{H}$ contains no red copy of $P_{\ell}$, so by Theorem~\ref{chavtal} it contains a blue copy of $K_{\frac{\beta N}{(\beta+1)\ell}}$. In particular, we have $r\ge\frac{\beta N}{(\beta+1)\ell}$. Now we use a double counting argument on $D = \sum_{w\in W}d_{red}(w;\mathbf{H})$ to obtain a lower bound on $W$. On the one hand, Observation~\ref{observe} tells us that there are at least
\[N-|L|-r\ge \frac{p}{2}+(\beta+1)\beta\ell\sqrt{m}\]
red edges coming out from $W$, so we have $D\ge \frac{p}{2} + (\beta+1)\beta\ell\sqrt{m}$. On the other hand, we have $d_{red}(v;\mathbf{H})\le(\beta+1)\ell$ for all $v\in V(\mathbf{H})$, so $D \le |W|(\beta+1)\ell$. Hence, we obtain
\begin{equation} \label{eq:W-lower}
|W|\ge \frac{\frac{p}{2}+(\beta+1)\beta\ell\sqrt{m}}{(\beta+1)\ell}\ge \beta\sqrt{m}\;.
\end{equation}
We shall also obtain an upper bound on $|W|$. Since the vertices of $G$ were enumerated in decreasing degree order, we have
\begin{equation} \label{eq:W-upper}
|W| = d_{G}(v_{r+1}) < \frac{2m}{r} \le \frac{2m(\beta+1)\ell}{\beta N} \le \frac{2\sqrt{m}}{\beta(\beta+1)}\;.
\end{equation}
Combining~\eqref{eq:W-lower} and~\eqref{eq:W-upper} gives $\beta^2(\beta+1)<2$, which contradicts $\beta\ge3$.

\vskip 2mm
{\bf Case 2.} 
$\frac{p}{2} < r < p$.

In this case, we have $|V(\mathbf{H})\setminus V(\mathbf{L})| > \beta(\beta+1)\ell\sqrt{m}$. By Observation~\ref{observe} and the pigeonhole principle, there is a vertex $v\in W$ with $d_{red}(v;\mathbf{H}) > \frac{\beta(\beta+1)\ell\sqrt{m}}{|W|}$. Since the decreasing degree order of $G$ gives $|W|\le2m/r$ and we have $d_{red}(v;\mathbf{H})\le(\beta+1)\ell$, we obtain
\[(\beta+1)\ell \ge d_{red}(v;\mathbf{H}) > \frac{\beta(\beta+1)\ell\sqrt{m}}{|W|} > \frac{\beta(\beta+1)\ell p}{4\sqrt{m}} > \frac{\beta(\beta+1)\ell}{2\sqrt{2}}\;,\]
where the final inequality uses $p>\sqrt{2m}$. But this gives a contradiction because $\beta\ge3$. This completes the proof of Lemma \ref{path--}.
\end{proof}

\subsection{Proof of Theorem \ref{cycle-epsion-p2}}
The following result will be used in the proof of Theorem \ref{cycle-epsion-p2}.
\begin{theorem}[\cite{Erdos-Faudree-Rousseau-Schelp1978}]\label{Erdos-Faudree-Rousseau-Schelp-1978}
For all $t\geq 3$ and $p\geq 2$, writing $\ell=\left\lfloor\frac{t-1}{2} \right\rfloor$, we have
\[r(C_{t},K_{p})\leq \left(t-2\right)p^{1+\frac{1}{\ell}}+(2t-3)p.\]
\end{theorem}

Equipped with these preliminary results, we turn to the proof of our main theorem.

Let $k\ge2$. Fix a sufficiently large constant $B_k\geq 20k^{2}$ so that for every graph $G$ on $p$ vertices with $m\le3^{20}$ edges and no isolated vertices, we have
\[r(C_{2k+1},G)\leq 2m\left(1+B_km^{-1/20}\right)+p\;.\]
We shall prove by induction on $p$ that every graph $G$ on $p$ vertices with $m$ edges and no isolated vertices has $r(C_{2k+1},G)\leq N := 2m\left(1+B_km^{-1/20}\right)+p$. Note that by our choice of $B_k$ this clearly holds whenever $m\le3^{20}$; in particular, this covers all $p\le3^{10}$ and gives a base case for the induction. Hence, we assume that $m>3^{20}$. Without loss of generality, we may also assume that $G$ is connected because $r(C_{2k+1},G_{1}\cup G_{2}) \leq \max\{r(C_{2k+1},G_{1}) + r(C_{2k+1},G_{2})\}$ holds for any disjoint union $G_1\cup G_2$ of graphs. We proceed with the following two cases.
%\setcounter{case}{0}
%\begin{case}
%$m\geq(2k-1)p^{1+\max\{\frac{1}{k},\frac{1}{19}\}}$. 
%\end{case}
\vskip 2mm
{\bf Case 1.} \emph{$m\geq(2k-1)p^{1+\max\{\frac{1}{k},\frac{1}{19}\}}$.}

By Theorem~\ref{Erdos-Faudree-Rousseau-Schelp-1978}, we have $r(C_{2k+1},G)\leq r(C_{2k+1},K_{p})\leq (2k-1)p^{1+\frac{1}{k}}+(4k-1)p\leq N.$

%\begin{case}
%$m<(2k-1)p^{1+\max\{\frac{1}{k},\frac{1}{19}\}}$.
%\end{case}
\vskip 2mm
{\bf Case 2.} \emph{$m<(2k-1)p^{1+\max\{\frac{1}{k},\frac{1}{19}\}}$.}

Take a red/blue-colored copy $\mathbf{K}$ of $K_{N}$ and suppose that it contains neither a red copy of $C_{2k+1}$ nor a blue copy of $G$. Obtain $G'$ from $G$ by deleting a minimum degree vertex $v$. Since $|V(G')|=p-1$, by the induction hypothesis $\mathbf{K}$ contains a blue copy $\mathbf{L}$ of $G'$. Let $X$ be the set of vertices of $\mathbf{K}$ not in $\mathbf{L}$ and $Y$ be the set of vertices in $\mathbf{L}$ corresponding to the neighbors of $v$ in $G$.
For $u\in Y$ write $R(u)$ for the set of vertices in $X$ joined to $u$ by a red edge. Note that each vertex $x\in X$ is joined by a red edge to at least one vertex in $Y$. Indeed, any vertex $x\in X$ with no red edge to $Y$ could be used to extend $\mathbf{L}$ to a blue copy of $G$, giving a contradiction. Hence, we have
\begin{equation}\label{e2}
X \subseteq \bigcup_{u\in Y}R(u)\;.
\end{equation}

We prove the following claim.
\begin{claim}\label{e3}
For all $u\in Y$ we have $|R(u)|< r(P_{2k},G)$.
\end{claim}
\begin{claimproof}
Suppose for a contradiction that $|R(u)|\geq r(P_{2k},G)$. Since $\mathbf{K}$ does not contain a blue copy of $G$, the lower bound on $|R(u)|$ implies that there is a red copy $\mathbf{P}$ of $P_{2k}$ on $R(u)$. But now $u$ completes $\mathbf{P}$ to a red copy of $C_{2k+1}$, giving a contradiction.
\end{claimproof}

By~\eqref{e2} and Claim~\ref{e3}, we have
\begin{equation*}
N  = (p-1)+|X| \leq (p-1)+\sum_{u\in Y}|R(u)| < p+|Y|\cdot r(P_{2k},G)\;.
\end{equation*}
Now by applying Lemma~\ref{path--} with $\beta=m^{1/20}\ge3$, we obtain
\begin{equation*}
N < p + |Y|p\left(1+m^{-1/20}\right)+2|Y|k(m^{1/20}+1)^{2}\sqrt{m}
\end{equation*}
Observe that the case condition gives $p^{-1}<\left(\frac{2k-1}{m}\right)^{1-\max\{\frac{1}{k+1},\frac{1}{20}\}}$ and we have $|Y|\le2m/p$ because $v$ is a minimum degree vertex of $G$. By applying these observations, we obtain
\begin{align*}
N & < p + |Y|p\left(1+m^{-1/20}\right) + 2|Y|k(m^{1/20}+1)^{2}\sqrt{m} \\
& \le p + 2m\left(1+m^{-1/20}+8km^{3/5}p^{-1}\right) \\
& \le p + 2m\left(1+m^{-1/20}+16k^2m^{-1/15}\right) \le N,
\end{align*}
which gives the required contradiction. This completes the proof of Theorem \ref{cycle-epsion-p2}. \hfill$\Box$

\subsection{Proof of Theorem~\ref{multi-kr}}
We proceed by double induction on $r$ and $p=|V(G)|$. For all $r\ge3$ let
\begin{equation} \label{eq:c_r}
c_{r} := 2^{r-1}\;.
\end{equation}
For the base case $r=3$, we have $r(K_{3},G)\leq 2m+1\leq 3m$ from a result of Sidorenko \cite{Sidorenko}. Now assume $r>3$ and that the theorem holds for smaller values of $r$. For the base case $p=2$, we have $r(K_{r},K_{2}) = r \leq c_r$. Now further assume $p>2$ and that the theorem holds for smaller values of $p$. Since $r(K_{r};G_{1}\cup G_{2})\leq r(K_{r};G_{1})+r(K_{r};G_{2})$ holds for any disjoint union $G_1\cup G_2$ of graphs, we may assume without loss of generality that $G$ is connected.

Take a red/blue-colored copy $\mathbf{K}$ of $K_{N}$ on $N\geq c_{r}m^{\frac{r-1}{2}}$ vertices. Suppose it has neither a red copy of $K_{r}$ nor a blue copy of $G$. Obtain $G'$ from $G$ by deleting a minimum degree vertex $v$. Since $|V(G')|=p-1$, by the induction hypothesis $\mathbf{K}$ contains a blue copy $\mathbf{L}$ of $G'$. Let $X$ be the set of vertices of $\mathbf{K}$ not in $\mathbf{L}$ and $Y$ be the set of vertices in $\mathbf{L}$ corresponding to the neighbors of $v$ in $G$.
For $u\in Y$ write $R(u)$ for the set of vertices in $X$ joined to $u$ by a red edge. Note that each vertex $x\in X$ is joined by a red edge to at least one vertex in $Y$. Indeed, we could use any vertex $x\in X$ with no red edge to $Y$ to extend $\mathbf{L}$ to a blue copy of $G$, giving a contradiction. Hence, we have
\begin{equation}\label{e8}
X \subseteq \bigcup_{u\in Y}R(u)\;.
\end{equation}

We prove the following claim.
\begin{claim}\label{label2}
For all $u\in Y$ we have $|R(u)|< r(K_{r-1},G)$.
\end{claim}
\begin{claimproof}
Suppose for a contradiction that $|R(u)|\geq r(K_{r-1},G)$. Since $\mathbf{K}$ does not contain a blue copy of $G$, the lower bound on $|R(u)|$ implies that there is a red edge on $R(u)$. But this together with $u$ gives a red copy of $K_3$, which gives a contradiction.
\end{claimproof}

By~\eqref{e8} and Claim~\ref{label2}, we obtain
\begin{equation*}
N  = (p-1)+|X| \leq (p-1)+\sum_{u\in Y}|R(u)| < p+|Y|\cdot r(K_{r-1},G)\;.
\end{equation*}
Since $v$ is a minimum degree vertex of $G$, we have $|Y|\le\sqrt{2m}$. By applying this together with the bound on $r(K_{r-1},G)$ from the inductive hypothesis, $p\le m+1$ and~\eqref{eq:c_r}, we get
\begin{equation*}
N < p + |Y|\cdot r(K_{r-1},G) \le p+\sqrt{2m}\cdot c_{r-1}m^{\frac{r}{2}-1} \le c_{r}m^{\frac{r-1}{2}} \le N\;,
\end{equation*}
which gives the required contradiction.
\hfill$\Box$
\subsection{Proof of Theorem~\ref{multi-k3}}
We proceed by double induction on $k$ and $p=|V(G)|\ge2$. For all $k\ge1$ let
\begin{equation} \label{eq:c_k}
    c_{k}:=3\cdot2^{k-1}\cdot k!\;.
\end{equation}
For the case $k=1$, we have $r_{2}(K_{3};G)\leq 2m+1\leq c_{1}m$ from a result of Sidorenko \cite{Sidorenko}. Now assume $k>1$ and that the theorem holds for smaller $k$. For the case $p=2$, we have $r_{k+1}(K_{3};K_{2}) = r_k(K_{3};K_{3})\leq 3\cdot k!\leq c_{k}$ from a result of Greenwood and Gleason~\cite{Greenwood-Gleason}. We further assume $p>2$ and that the theorem holds for smaller $p$. Without loss of generality, we may assume that $G$ is connected because $r_{k+1}(K_{3};G_{1}\cup G_{2})\leq r_{k+1}(K_{3};G_{1})+r_{k+1}(K_{3};G_{2})$ holds for any disjoint union $G_1\cup G_2$ of graphs.

For $N\geq c_{k}m^{\frac{k+1}{2}}$ take a copy $\mathbf{K}$ of $K_{N}$ edge-colored with colors $1,\dots,k+1$. Supposethat it contains neither a monochromatic copy of $K_{3}$ in some color $i\in\{1,\ldots,k\}$ nor a copy of $G$ in color $k+1$. Let $G'$ be obtained from $G$ by deleting a minimum degree vertex $v$. Since $|V(G')|=p-1$, by the induction hypothesis $\mathbf{K}$ contains a copy $\mathbf{L}$ of $G'$ in color $k+1$. Let $X$ be the set of vertices of $\mathbf{K}$ not in $\mathbf{L}$ and $Y$ be the set of vertices in $\mathbf{L}$ corresponding to the neighbors of $v$ in $G$. 
% \begin{align}\label{e4}
% N=(p-1)+|X|.
% \end{align}
For $1\le i\le k$ and $u\in Y$ write $W_{i}(u)$ for the set of vertices in $X$ joined to $u$ by an edge in color $i$. Note that each vertex $x\in X$ is joined to at least one vertex in $Y$ by an edge not in color $k+1$. Indeed, we could use any vertex $x\in X$ for which this does not hold to extend $\mathbf{L}$ to a copy of $G$ in color $k+1$, giving a contradiction. Hence, we have
\begin{equation}\label{e5}
X \subseteq \bigcup_{u\in Y}\left(\bigcup_{i\in \{1,\ldots,k\}}W_{i}(u)\right).
\end{equation}

We have the following claim.
\begin{claim}\label{label1}
For all $i\in\{1,\dots,k\}$ and $u\in Y$ we have $|W_{i}(u)|< r_{k}(K_{3};G)$.
\end{claim}
\begin{claimproof}
Suppose for a contradiction that $|W_{i}(u)|\geq r_{k}(K_{3};G)$. Since $\mathbf{K}$ contains neither a copy of $G$ in color $k+1$ nor a monochromatic copy of $K_{3}$ in some color $j\in\{1,\ldots,k\}\setminus\{i\}$, the lower bound on $|W_{i}(u)|$ implies that there is an edge on $W_{i}(u)$ with color $i$. But this together with $u$ gives a copy of $K_3$ in color $i$, which gives a contradiction.
\end{claimproof}
By~\eqref{e5} and Claim~\ref{label1}, we obtain
\begin{equation*}
N  = (p-1)+|X| \leq (p-1)+\sum_{u\in Y}\left(\sum_{i\in \{1,\ldots,k\}}|W_{i}(u)|\right) < p+|Y|\cdot k\cdot r_{k}(K_{3};G)\;.
\end{equation*}
Since $v$ is a minimum degree vertex of $G$, we have $|Y|\le\sqrt{2m}$. By applying this together with the bound on $r_{k}(K_{3};G)$ from the inductive hypothesis, $p\le m+1$ and~\eqref{eq:c_k}, we get
\begin{equation*}
N < p + |Y|\cdot k\cdot r_{k}(K_{3};G) \le p+\sqrt{2m}\cdot k\cdot c_{k-1}m^{\frac{k}{2}} \le c_{k}m^{\frac{k+1}{2}} \le N\;,
\end{equation*}
which gives the required contradiction.\hfill$\Box$

\section{Concluding remarks}
It is worth mentioning that
Erd\H{o}s, Faudree, Rousseau and Schelp asked the following Question \ref{Q2} which is to extend the result of Sidorenko on triangles to cycles of arbitrary length in the following, which was confirmed if $k=3,4,5$ in \cite{Radziszowski}. Cambie, Freschi, Morawski, Petrova and Pokrovskiy answered Question \ref{Q2}, provided $m$ is sufficiently large with respect to $k$ in \cite{Cambie-Freschi-Morawski-Petrova-Pokrovskiy}.
\begin{question}[\cite{Erdos-Faudree-Rousseau-Schelp}]\label{Q2}
Is $r(C_{k},G)\leq 2m+\lfloor\frac{k-1}{2}\rfloor$, where $k\geq 3$ and $G$ is a graph of size $m$ without isolated vertices?
\end{question}

Erd\H{o}s, Faudree, Rousseau and Schelp also asked the following question in \cite{Erdos,Erdos-Faudree-Rousseau-Schelp}.
\begin{question}[\cite{Erdos,Erdos-Faudree-Rousseau-Schelp}]\label{Q3}
Is $K_{3, 3}$ Ramsey size linear?
\end{question}

\noindent {\bf Acknowledgement.} The authors gratefully acknowledge the many helpful suggestions and discussions of Professor Hong Liu during the preparation of the paper. The second author also wishes to thank the ECOPRO group for their warm hospitality. The authors also wish to express their gratitude to Yulai Ma for his careful reading of the manuscript.
\vskip 2mm
\noindent {\bf Declaration of Competing Interests.}
The authors declare that they have no known competing financial interests or personal relationships that could have appeared to influence the work reported in this paper.
\bibliography{biblio}
\bibliographystyle{abbrv}

\end{document}